\documentclass[a4paper,11pt,reqno]{amsart}

\usepackage{hyperref} 
\hypersetup{
pdfauthor={Jan Derezi\'nski, Micha\l\ Wrochna},
pdftitle={Continuous and holomorphic functions with values in closed operators}, 
pdfborder={0 0 0}    
}
\usepackage{amsmath,amssymb,epsf}
\usepackage{amsfonts,amsbsy,amscd}
\usepackage{paralist}
\usepackage{latexsym,amsopn}
\usepackage{verbatim}
\usepackage{amsthm}
\usepackage{times}
\usepackage{fancyhdr}
\usepackage{color}
\newcommand{\qeds}{{\qed\smallskip }}
\newcommand{\red}{}

 \newtheorem{thm}{Theorem}[section]
  \newtheorem{theorem}[thm]{Theorem}
 \newtheorem{corollary}[thm]{Corollary}
 \newtheorem{lemma}[thm]{Lemma}
 \newtheorem{proposition}[thm]{Proposition}
 \newtheorem{definition}[thm]{Definition}
 \theoremstyle{remark}
 \newtheorem{remark}[thm]{Remark}
 \newtheorem{example}[thm]{Example}
 \numberwithin{equation}{section}
\newcommand{\com}{{\rm com}}
\newcommand{\inv}{{\rm inv}}
\newcommand{\linv}{{\rm linv}}
\newcommand{\Dom}{{\rm Dom}}
\newcommand{\Grass}{{\rm Grass}}
\newcommand{\Ker}{{\rm Ker}}
\newcommand{\Gr}{{\rm Gr}}
\newcommand{\Pro}{{\rm Pr}}
\newcommand{\Ran}{{\rm Ran}}

\renewcommand{\ker}{{\rm Ker}}
\newcommand{\bra}{\langle} 
\newcommand{\ket}{\rangle}

\renewcommand\Re{{\rm Re\,}}

\renewcommand{\max}{{\rm max}}

\newcommand{\p}{{\rm p}}

\renewcommand{\sp}{{\rm sp}}
\newcommand{\rs}{{\rm rs}}

\newcommand{\grintl}{[\kern-.18em [}
\newcommand{\grintr}{]\kern-.18em ]}

\def\Ran{{\rm Ran\,}}
\def\ran{{\rm Ran\,}}
\def\dist{{\rm dist\,}}

\def\one{{\mathchoice {\rm 1\mskip-4mu l} {\rm 1\mskip-4mu l} {\rm
      1\mskip-4.5mu l} {\rm 1\mskip-5mu l}}} 

\newcommand{\mat}[4]{\left[\begin{array}{cc}#1 &#2  \\ #3 &#4 \end{array}\right]}

 \def\cB{{\mathcal B}} \def\cC{{\mathcal C}}
  
 \def\cH{{\mathcal H}} 
 \def\cK{{\mathcal K}} \def\cL{{\mathcal L}}

  \def\cX{{\mathcal X}}
\def\cY{{\mathcal Y}} 

\newcommand{\R}{{\mathbb R}}

\newcommand{\C}{{\mathbb C}}

\newcommand{\ben}{\begin{enumerate}}
\newcommand{\een}{\end{enumerate}}
\newcommand{\beq}{\begin{equation}}
\newcommand{\eeq}{\end{equation}}
\newcommand{\bep}{\begin{proposition}}
\newcommand{\eep}{\end{proposition}}
\newcommand{\bed}{\begin{definition}}
\newcommand{\eed}{\end{definition}}
\newcommand{\beqa}{\begin{eqnarray}}
\newcommand{\eeqa}{\end{eqnarray}}

\newcounter{smallroman}
\newenvironment{romanenumerate}
{\begin{list}{{\normalfont\textrm{(\roman{smallroman})}}}
  {\usecounter{smallroman}\setlength{\itemindent}{0cm}
   \setlength{\leftmargin}{5ex}\setlength{\labelwidth}{4ex}
   \setlength{\topsep}{0.75\parsep}\setlength{\partopsep}{0ex}
   \setlength{\itemsep}{0ex}}}
{\end{list}}

        {\begin{compactitem}[#1]}{\end{compactitem}}

\begin{document}

\title{Continuous and holomorphic functions with values in closed operators}
\author{\textsc{Jan Derezi\'nski$^{1}$,\  Micha\l\ Wrochna$^{2}$}}

\date{}

\maketitle
\begin{center}\small
$^{1}$Department of Mathematical Methods in Physics,\\ Faculty of Physics, 
University of Warsaw\\ Ho\.{z}a 74, Warsaw
00-682, Poland\\ \vspace{0.2cm}
$^{2}$D\'epartement de Math\'ematiques, Universit\'e Paris-Sud XI,\\ B\^at. 425, 91405 Orsay Cedex, France\\
\vspace{0.2cm} e-mail: $^{1}$\texttt{Jan.Derezinski@fuw.edu.pl}, \ \ $^{2}$\texttt{Michal.Wrochna@math.u-psud.fr} \vspace{0.5cm} \end{center}

\begin{abstract} We systematically derive general properties of continuous and holomorphic functions with values in closed operators, allowing in particular for operators with empty resolvent set. We provide criteria for a given operator-valued  function to be continuous or holomorphic. This includes sufficient conditions for the sum and product of operator-valued holomorphic functions  to be holomorphic.

Using graphs of operators, operator-valued functions are identified with functions with values in subspaces of a Banach space. A special role is thus played by projections onto closed subspaces of  a Banach space, which depend holomorphically on a parameter.  
\end{abstract}

\vspace{0.5cm}


\section{Introduction}


\subsection*{Definition of continuous and holomorphic operator-valued functions}

It is obvious how to define the concept of a (norm) continuous or holomorphic function with values in bounded operators.
 In fact,
let
 $\cH_1$, $\cH_2$ be Banach spaces.
 The set of bounded operators from $\cH_1$ to $\cH_2$, denoted
by $\cB(\cH_1,\cH_2)$, has a natural metric given by the norm, which can be used to define the continuity. For the holomorphy, we could use the following definition:
\begin{definition}
A function $\C\supset\Theta\ni z\mapsto T_z\in \cB(\cH_1,\cH_2)$ 
 is holomorphic if $\lim\limits_{h\to0}\frac{T_{z+h}-T_z}{h}$ exists for all $z\in\Theta$.\label{defo1}\end{definition}
There exist other equivalent definitions. For instance,  we can demand that $z\mapsto \langle y |T_z x\rangle$ is holomorphic for all  bounded linear functionals $y$ on $\cH_2$ and all $x\in\cH_1$. 

It is easy to see that holomorphic functions with values in bounded operators have good properties that generalize the corresponding properties of $\C$-valued functions.
 For instance, the product of holomorphic functions is holomorphic; we
 have the uniqueness of the holomorphic continuation; if $z\mapsto
 T_z$ is holomorphic, then so is $z\mapsto T_{\bar z}^*$ (the ``Schwarz
 reflection  principle'').

 In practice, however, especially in mathematical physics and partial differential equations, one often  encounters functions with values in unbounded closed operators, for which the continuity and the holomorphy are more tricky. In our paper we  collect and prove various general facts concerning continuous and, especially, holomorphic functions with values in closed operators. In particular, in the context of Hilbert spaces we give a certain necessary and sufficient criterion for the holomorphy, which appears to be new and useful, and was the original motivation for writing this article.
Besides, we provide sufficient conditions for the continuity and holomorphy of the product and sum of operator-valued functions. 

Our main motivation is the case of Hilbert spaces. However, the main tool that we use are non-orthogonal projections, where the natural framework is that of Banach spaces, which we use for the larger part of our paper. In particular,  we  give  a systematic discussion of  elementary properties of projections on a Banach space. Some of them  we have not seen in the literature.

The continuity and holomorphy of functions with values in closed operators 
is closely related to the continuity and holomorphy of functions with values in closed subspaces. The family of closed
subspaces of a Banach space $\cH$ will be called its {\em  Grassmannian} and denoted $\Grass(\cH)$. It possesses a natural metric topology given by the  {\em gap function} {\red (Def \ref{def-gap})}.  

{\red  In order to make this introduction reasonably readable,  let us fix some terminology (which we give independently and in more detail
in the main part of  the article). By a {\em bounded left invertible operator} we  mean a bounded injective operator with a closed range (Def. \ref{def-inv}). For a function $z\mapsto\cK_z\in\Grass(\cH)$, its {\em injective resolution} is a function $z\mapsto T_z\in \cB(\cH_1,\cH)$ whose values are left-invertible and $\Ran T_z=\cK_z$ (Def. \ref{def-resol}).

The following proposition (Prop. \ref{pro1a} in the main text)
gives a useful  characterization of functions continuous in the gap topology. 
\begin{proposition} If $z\mapsto\cK_z\in \Grass(\cH)$ possesses a continuous injective resolution, then it is continuous in the gap topology. \label{prop1}\end{proposition}

It is more tricky to define the holomorphy of a Grassmannian-valued function, than to define its continuity (Def. \ref{def-holo}):
\begin{definition}
 $z\mapsto\cK_z\in \Grass(\cH)$ is holomorphic if locally it possesses a holomorphic injective resolution.
\label{def1}\end{definition}
Arguably, this definition seems less satisfactory than that of the (gap) continuity --- but we do not know of any better one.}

Let  $\cC(\cH_1,\cH_2)$ denote the set of closed operators from $\cH_1$ to $\cH_2$. 
 It is natural to ask what is the natural concept of continuity and holomorphy for functions with values in  $\cC(\cH_1,\cH_2)$.

By identifying a closed operator with its graph we transport the gap topology from  $\Grass(\cH_1\oplus\cH_2)$ to $\cC(\cH_1,\cH_2)$.  This yields a natural definition of the continuity of functions with values in closed operators.
There are other possibilities, but the {\em gap topology} seems to be the most natural generalization of the norm topology from $\cB(\cH_1,\cH_2)$ to
 $\cC(\cH_1,\cH_2)$.

{\red We have the following criterion for the continuity of operator valued functions (Prop. \ref{propi}):
\begin{proposition} Consider a function $\Theta\ni z\mapsto T_z\in\cC(\cH_1,\cH_2)$. If there exists a neighbourhood $\Theta_0\subset\Theta$ of $z_0$, a Banach space $\cK$ and a continuous function $\Theta_0\ni z\mapsto W_{z}\in \cB(\cK,\cH_1)$, such that $W_z$ maps bijectively $\cK$ onto $\Dom(T_{z})$ for all $z\in\Theta_0$ and
\[
\Theta_0\ni z\mapsto T_{z}W_{z}\in \cB(\cK,\cH_2)
\]
is continuous, then
 $\Theta\ni z\mapsto T_z\in\cC(\cH_1,\cH_2)$ in the gap topology.
\label{pro2}\end{proposition}
The function  $z\mapsto W_z$ will  be called a {\em resolution of continuity} of $z\mapsto T_z$.}




{\red Similarly, we transport the holomorphy from  $\Grass(\cH_1\oplus\cH_2)$ to $\cC(\cH_1,\cH_2)$, which provides a definition
 of a holomorphic function with values in closed operators (Def \ref{def-holo-op}).
 This definition,  strictly speaking due to T.~Kato \cite{K},
goes back essentially to F.~Rellich \cite{R} and can be 
{ reformulated} as follows (Prop. \ref{func7}):}
\begin{definition}\label{def:hol} A function $\Theta\ni z\mapsto T_z\in\cC(\cH_1,\cH_2)$ is holomorphic around $z_0\in\Theta$ if there exists a neighbourhood $\Theta_0\subset\Theta$ of $z_0$, a Banach space $\cK$ and a holomorphic function $\Theta_0\ni z\mapsto W_{z}\in \cB(\cK,\cH_1)$, such that $W_z$ maps bijectively $\cK$ onto $\Dom(T_{z})$ for all $z\in\Theta_0$ and
\[
\Theta_0\ni z\mapsto T_{z}W_{z}\in \cB(\cK,\cH_2)
\]
is holomorphic.
\label{defo2}\end{definition}

At first glance Def. \ref{defo2} may seem somewhat artificial, especially when compared with Def. \ref{defo1}, which looks as natural as possible. In particular, it involves a relatively arbitrary function $z\mapsto W_z$, { which will be called a {\em resolution of holomorphy} of $z\mapsto T_z$}.
The arbitrariness of { a  resolution of holomorphy}  is rarely a practical problem, because there exists a convenient criterion which works when {$\cH_1=\cH_2$ and} $T_z$ has a non-empty resolvent set (which is usually the case in applications). It is then enough to check the holomorphy of the resolvent of $T_z$. This criterion is however useless for $z\in\Theta$ corresponding to $T_z$ with an empty resolvent set, { or simply when $\cH_1\neq\cH_2$}. In this case, at least in the context of Hilbert spaces, our criterion given in Prop. \ref{qeq9} could be {particularly} useful.

{\red
Note that we use the word {\em resolution}  in two somewhat different contexts -- that of the Grassmannian (Prop. \ref{prop1} and Def. \ref{def1}) and that of closed operators (Prop. \ref{pro2} and Def. \ref{defo2}). Moreover, in the case of the continuity the corresponding object  gives a criterion
(Props. \ref{prop1} and  \ref{pro2}), whereas  in the case of the holomorphy it provides a definition (Defs. \ref{def1} and \ref{defo2}).} 

{\red The concept of the holomorphy for closed operators seems more complicated  than that of the continuity, and also  than the holomorphy of bounded operators.
Even the seemingly simpliest questions, such as the unique continuation or the validity of the Schwarz reflection principle, are somewhat tricky to prove.}

\subsection*{Examples} As an illustration, let us give a number of
 examples of holomorphic functions with values in closed operators. 
\begin{enumerate}
\item Let $z\mapsto T_z\in\cC(\cH_1,\cH_2)$ be a function. Assume that $\Dom(T_z)$ does not depend on $z$ and $T_z x$ is holomorphic for each $x\in \Dom(T_z)$. Then such function $z\mapsto T_z$ is holomorphic and it is called  a \emph{holomorphic family of type A}. Type A families inherit many good properties from Banach space- valued holomorphic functions and provide the least pathological class of examples.
\item\label{ex2}
Let $A\in\cC(\cH)$ have a nonempty resolvent set. Then $z\mapsto (A-z\one)^{-1}\in\cB(\cH)$ is holomorphic away from the spectrum of $A$. However, $z\mapsto (A-z\one)^{-1}\in\cC(\cH)$ is  holomorphic away from the point spectrum of $A$, see Example \ref{exa+}.  This shows that a nonextendible holomorphic function with values in bounded operators can have an extension when we allow unbounded values.
\item
Consider the so-called Stark Hamiltonian on $L^2(\R)$
\[H_z:=-\partial_x^2+zx.\]
(Here $x$ denotes the variable in $\R$ and $z$ is a complex parameter). For $z\in\R$ one can define $H_z$ as a self-adjoint operator with $\sp H_z=\R$. $H_z$ is also naturally defined for 
 $z\in\C\backslash\R$, and then has an empty spectrum { \cite{herbst}}. Thus in particular all non-real numbers belong to the resolvent set of $H_z$. One can show that $z\mapsto H_z$ is holomorphic only outside of the real line. On the real line it is even not continuous.
\item Consider the Hilbert space $L^2([0,\infty[)$. For $m>1$ the Hermitian operator
\beq H_m:=-\partial_x^2+\Big(m^2-\frac14\Big)\frac1{x^2}\label{defo3}\eeq
is essentially self-adjoint on $C_{\rm c}^\infty(]0,\infty[)$.  We can continue holomorphically (\ref{defo3}) onto the halfplane $\{\Re m>-1\}$. For all such $m$ the spectrum of  $H_m$ is $[0,\infty[$. In \cite{BDG} it was asked whether $m\mapsto H_m$ can be extended to the left of the line $\Re m=-1$. If this is the case, then on this line the spectrum will cover the whole $\C$ and the point $m=-1$ will be a singularity. We still do not know what is the answer to this question. We hope that the method developed in this paper will help to solve the above problem.
\end{enumerate}

\subsection*{Main results and structure of the paper} We start by introducing 
in Sec. \ref{sec2} the basic definitions and facts on (not necessarily orthogonal) projections in a linear algebra context.

Sec. \ref{sec3} contains the essential part of the paper. We first recall the definition of the gap topology on the Grassmanian and demonstrate how it can be characterized in terms of continuity of projections. It turns out that a significant role is played by the assumption that the subspaces are complementable. We then discuss holomorphic functions with values in the Grassmanian and explain  how this notion is related to Rellich's definition of operator-valued holomorphic functions. Most importantly, we deduce a result on the validity of the Schwarz reflection principle in Thm \ref{cor:adjointhol} and we recover Bruk's result on the uniqueness of analytic continuation in Thm. \ref{theorem:bruk}.

In Sec. \ref{sec4} we consider the case of operators on Hilbert spaces. We derive explicit formulae for projections on the graphs of closed operators and deduce a criterion for the holomorphy of operator-valued functions, which is also valid for operators with empty resolvent set (Prop. \ref{qeq9}).

In Sec. \ref{sec5} we give various sufficient conditions for the continuity and holomorphy of the product and sum of operator-valued functions. More precisely, we assume that $z\mapsto A_z, B_z$ are holomorphic and $A_z B_z$ is closed. 
The simpliest cases when $z\mapsto A_z B_z$ is holomorphic are discussed in Prop. \ref{criterion1} ($A_z$ boundedly invertible or $B_z$ bounded) and Prop. \ref{criterion2} ($A_z B_z$ densely defined and $B_{\bar z}^* A_{\bar z}^*$ closed). More sufficient conditions are given in Thm. \ref{theorem:holoproduct} ($\rs(A_zB_z)\cap\rs(B_z A_z)\neq\emptyset$) and Thm. \ref{thm:rancase} (${\rm Dom}(A_z)+\ran B_z=\cH$). A result on sums is contained in Thm. \ref{thm:sum}.

\subsection*{Bibliographical notes} The standard textbook reference for continuous and holomorphic operator-valued function is the book of T.~Kato \cite{K}. Most of the results are however restricted to either holomorphic families of type A or to operators with non-empty resolvent set. The first proof of the uniqueness of a holomorphic continuation outside of these two classes is due to V.M.~Bruk \cite{B}. The strategy adopted in our paper is to a large extent a generalization of \cite{B}.

Holomorphic families of subspaces were introduced by M.~Shubin \cite{shubin}; the definition was then reworked by I.C.~Gohberg and J.~Leiterer, see \cite{GL} and references therein. We use the definition from \cite[Ch.6.6]{GL}. It is worth mentioning that the original motivation for considering families of subspaces depending holomorphically on a parameter comes from problems in bounded operator theory, such as the existence of a holomorphic right inverse of a given function with values in right-invertible bounded operators, see for instance \cite{kab,LR} for recent reviews. 

The gap topology was investigated by many authors, eg. \cite{berkson,DRW,GM,GL}, however some of the results that we obtain using non-orthogonal projections appear to be new. A special role in our analysis of functions with values in the Grassmannian of a Banach space is played by subspaces that possess a complementary subspace, a review on this subject can be found in \cite{KM}.

\subsection*{Limitations and issues}The holomorphy of functions with values in closed operators is a nice 
and natural concept. We are, however, aware of some limitations of its practical value. Consider for instance the Laplacian $\Delta$ on $L^2(\R^d)$. As discussed in Example (\ref{ex2}), the resolvent $z\mapsto(z\one+\Delta)^{-1}$ extends to an entire holomorphic function. On the other hand, for many practical applications to spectral and scattering theory of Schr\"odinger operators another fact is much more important. Consider, for example, odd $d$ and $f\in C_{\rm c}(\R^d)$. Then \[z\mapsto f(x)(z\one+\Delta)^{-1}f(x)\] extends to a multivalued holomorphic function, and to make it single valued, one needs to define it on the Riemann surface of the square root (the double covering of $\C\backslash\{0\}$). 
The extension of this function to the second sheet of this Riemann surface (the so called {\em non-physical sheet of the complex plane}) plays an important role in the theory of resonances (cf. eg. \cite{zworski}). It is however different from what one obtains from the extension of $z\mapsto(z\one+\Delta)^{-1}$ in the sense of holomorphic functions with values in closed operators.


Further issues are due to the fact that typical assumptions considered, eg., in perturbation theory, do not allow for a good control of the holomorphy. For instance,
we  discuss situations where seemingly natural assumptions on $T_z$ and $S_z$ do not ensure that the product $T_z S_z$ defines a holomorphic function. 

\subsection*{Applications and outlook}The main advantage of the holomorphy in the sense of Definition \ref{defo2} is that it uses only the basic structure of the underlying Banach space (unlike in the  procedure discussed before on the example of the resolvent of the Laplacian). 

Despite  various problems that can appear in the general case, we conclude from our analysis that there are classes of holomorphic functions which enjoy particularly good properties. This is for instance the case for functions whose values are Fredholm operators. We prove in particular that the product of two such functions functions is again holomorphic. In view of this result it is worth mentioning that the Fredholm analytic theorem (see e.g. \cite[Thm. D.4]{zworski}), formulated usually  for bounded operators, extends directly to the unbounded case. It seems thus interesting to investigate further consequences of these facts.

On a separate note, we expect that in analogy to the analysis performed in \cite{BDG} for the operator $-\partial_x^2+z{x^{-2}}$, specific operator-valued holomorphic functions should play a significant role in the description of self-adjoint extensions of exactly solvable Schr\"odinger operators, listed explicitly in \cite{DW} in the one-dimensional case. 

A problem not discussed here is the holomorphy of the closure of a given function with values in (non-closed) unbounded operators. Such problems often appear in the context of products of holomorphic functions with values in closed operators, and one can give many examples when the product has non-closed values, but the closure yields a holomorphic function. A better understanding of this issue could lead to useful improvements of the results of the present paper.\medskip

\textbf{Acknowledgments.} 
 The research of J.D.  was supported in part by the National Science Center (NCN) grant No. 2011/01/B/ST1/04929.

\section{Linear space theory}\label{sec2}

Throughout this section 
 $\cK,{\red \cK'},\cH$ are linear spaces.

\subsection{Operators}

\begin{definition} 
{ By a {\em linear operator} $T$ {\em from} $\cK$ {\em to} $\cH$ (or simply an {\em operator on $\cK$}, if $\cK=\cH$) we will mean a {\em
  linear function  $T:\Dom T\to\cH$}, where $\Dom T$ is a linear
subspace of $\cK$.} In the usual way we define its kernel $\Ker T$, which is a subspace of $\Dom T$ and its range $\Ran T$, which is 
a subspace of $\cH$. 
{ If $\Dom T=\cK$, then we will say that $T$ is {\em everywhere defined}.} \end{definition}

\begin{definition} 
{ We will write $\cL(\cK,\cH)$ for the space of linear everywhere defined operators from $\cK$ to $\cH$. } We set $\cL(\cH):=\cL(\cH,\cH)$. \end{definition}

\begin{definition}\label{jot}
 If $\cX$ is a subspace of $\cH$, let $J_\cX$ denote the embedding of $\cX$ into $\cH$. 
\end{definition}

\begin{definition}
 We will write $\tilde T$ for the operator $T$ understood as 
 an operator on $\cK$ to $\Ran T$.
\end{definition}

Clearly, $T=J_{\Ran T}\tilde T.$

If an operator $T$ { from $\cK$ to $\cH$} is injective we can define the operator
$T^{-1}$ on $\cH$ to $\cK$ with $\Dom T^{-1}=\Ran T$ and $\Ran T^{-1}=\Dom T$.
Clearly,
\[T^{-1}T=\one_{\Dom
  T},\ \ T
T^{-1}=\one_{\Ran T}.\]

Often, instead of $T^{-1}$ we will prefer to use
 $\tilde
T^{-1}:\Ran T\to\cK$, whose advantage is that it is everywhere defined.

\begin{definition} 
If $T$,$S$ are two operators, their  product $TS$ is defined in the usual way on the domain
\[
\Dom(TS)=S^{-1}\Dom(T)=\{ v\in \Dom(S) : \ S v\in\Dom(T) \}.
\]
(The notation `$S^{-1}\Dom(T)$' is understood as in the last equality
above, so that $S$ is not required to be injective).  
The operator $T+S$ is defined in the obvious way with  $\Dom(T+S)=\Dom T\cap\Dom S$.
\label{prodi}\end{definition}

\subsection{Projections}

\begin{definition} Let $\cX,\cY$ be two subspaces of $\cH$ with $\cX\cap\cY=\{0\}$. (We do not require that $\cX+\cY=\cH$.) We will write $ P_{\cX,\cY}\in\cL(\cX+\cY)$ for the
idempotent with range $\cX$ and kernel $\cY$.
 We will say that it is the {\em projection onto $\cX$ along $\cY$}.
\eed

\bed
If  $\cX\cap\cY=\{0\}$ and
 $\cX+\cY=\cH$, 
 we will say that $\cX$ is {\em complementary} to $\cY$, and write $\cX\oplus\cY=\cH$.
In such a case $P_{\cX,\cY}\in\cL
(\cH)$. 
\end{definition} 

 Clearly
 $\one-P_{\cX,\cY}=P_{\cY,\cX}$.

The following proposition  describes a useful formula for the  projection corresponding to a pair of subspaces: 
\begin{proposition}\label{pwo}
Let  $J\in\cL({\red \cK'},\cH)$ be  injective. Let $I\in\cL(\cH,\cK)$ be surjective. Then $IJ$ is bijective iff 
\[\Ran J\oplus \Ker I=\cH,\]
and then we have the following formula for
 the projection onto $\Ran J$ along $\Ker I$:
\begin{equation}
P_{\Ran J,\Ker I}=J(IJ)^{-1}I\label{pq2}
\end{equation}\end{proposition}
\proof $\Rightarrow$: $J(IJ)^{-1}$ is injective. Hence
$\Ker J(IJ)^{-1}I=\Ker I$.

 $(IJ)^{-1}I$ is surjective. Hence
$\Ran J(IJ)^{-1}I=\Ran J$.
Clearly,
\[ (J(IJ)^{-1}I)^2= J(IJ)^{-1}I.\]

$\Leftarrow$: Suppose that $IJ$ is not injective. Clearly, $J\Ker IJ\subset \Ker I\cap\Ran J$. But $J$ is injective. Hence, 
$\Ker I\cap\Ran J\neq\{0\}$.

Suppose that $IJ$ is not surjective. Clearly, $\Ran IJ=I(\Ran J+\Ker I)$. Since $I$ is surjective, 
$\Ran J+\Ker I\neq\cH$.\qeds

\begin{proposition}\label{pro3} Let $\cX,\cX',\cY$ be subspaces of $\cH$, with $\cX$ complementary to $\cY$. Then  $\tilde P_{\cX,\cY}J_{\cX'} $ is bijective iff
$\cX'$ and $\cY$ are complementary,
and then 
we have the following formula for  the projection onto $\cX'$ along $\cY$:
\begin{equation}
P_{\cX',\cY} =J_{\cX'}(\tilde P_{\cX,\cY}J_{\cX'})^{-1}\tilde P_{\cX,\cY}.\label{pq}\end{equation}
\end{proposition}{\red\proof $J_{\cX'}\in \cL(\cX',\cH)$ is injective and $\Ran J_{\cX'}=\cX'$. $\tilde P_{\cX,\cY}\in \cL(\cH,\cX)$ is surjective and $\Ker \tilde P_{\cX,\cY}=\cY$. So the proposition is a special case of Prop. \ref{pwo}. \qed}

\section{Banach space theory}\label{sec3}

Throughout this section 
 $\cK,{\red\cK',}\cH,\cH_1,\cH_2$ are Banach spaces. We will use the notation  $\cX^{\rm cl}$ for the closure of a {subset} $\cX$. { Similarly, for a closable operator $T$, we use the notation $T^{\rm  cl}$ for its closure.

 $\cB(\cK,\cH)$ will denote the space of bounded everywhere defined operators from $\cK$ to $\cH$. $\cC(\cK,\cH)$ will denote the space of closed operators from $\cK$ to $\cH$. We will write $\cB(\cH)=\cB(\cH,\cH)$ and $\cC(\cH)=\cC(\cH,\cH)$.}

\subsection{Closed range theorem}

Below we recall 
one of the most useful theorems of operator theory:
\begin{theorem}[Closed range theorem]
Let $T$ be an injective operator { from $\cK$ to $\cH$.} Then the following are equivalent:
\begin{enumerate}
\item $ \Ran T$ is closed and $T\in\cC(\cK,\cH)$;
\item $\Ran T$ is closed    and $\tilde T^{-1}$ is bounded;
\item  $\Ran T$ is closed    and \beq
\inf\limits_{x\neq0}\frac{\|Tx\|}{\|x\|}>0.\label{pwq}\eeq
\end{enumerate}
Besides, the number defined in (\ref{pwq}) is $\|\tilde T^{-1}\|^{-1}$.
\label{linv}
\end{theorem}

\begin{definition} Operators satisfying the conditions of Thm \ref{linv} will be called {\em left invertible}. The family of such operators
will be denoted by
$\cC_\linv(\cK,\cH)$. We will write 
$\cB_\linv(\cK,\cH)=\cB(\cK,\cH)\cap
\cC_\linv(\cK,\cH)$.

 {An} {operator $T$ satisfying the conditions of Thm \ref{linv} and such that   $\Ran T=\cH$ is called {\em invertible}.}
The family of such operators
will be denoted by
$\cC_\inv(\cK,\cH)$. We will write 
$\cB_\inv(\cK,\cH)=\cB(\cK,\cH)\cap
\cC_\inv(\cK,\cH)$.
\label{def-inv}
\end{definition}
Clearly, $T$ is left invertible iff $\tilde T$ is invertible. The next proposition shows that left invertibility is stable under bounded perturbations.

\begin{proposition}
Let $T\in\cC_\linv(\cK,\cH)$, $S\in\cB(\cK,\cH)$ and
$\|S \|<\|{\tilde T^{-1}}\|^{-1}$. Then  $T+S\in\cC_\linv(\cK,\cH)$, and
\beq \Big|\|\widetilde{(T+S)}^{-1}\|^{-1}-\|\tilde
T^{-1}\|^{-1}\Big|\leq\|S\|.\label{wew}\eeq 
Consequently,  $\cB_\linv(\cK,\cH)\ni T\mapsto\|\tilde T^{-1}\|$ is a
continuous function and $\cB_\linv(\cK,\cH)$ is an open subset of
$\cB(\cK,\cH)$. 
\end{proposition}
\proof We have the lower bound
\[\|(T+S)x\|\geq\|Tx\|-\|Sx\|\geq(\|\tilde
T^{-1}\|^{-1}-\|S\|)\|x\|.\]
But $\|\tilde
T^{-1}\|^{-1}-\|S\|>0$, therefore, $T+S$ is left invertible and
\[\|\widetilde{(T+S)}^{-1}\|^{-1}\geq\|\tilde T^{-1}\|^{-1}-\|S\|.\]
Then we switch the roles of $T$ and $T+S$, and obtain
\[\|\tilde T^{-1}\|^{-1}\geq\|\widetilde{(T+S)}^{-1}\|^{-1}-\|S\|,\]
which proves (\ref{wew}). \qed

\bed { The {\em resolvent set} of an operator $T$ on $\cH$, denoted $\rs\,T$, is defined to be the set of all $\lambda\in\C$ such that $T-\lambda\one\in\cC_\inv(\cH)$.}
 The {\em spectrum} of $T$ is by definition the set $\sp\,T:=\C\setminus \rs\,T$.
\eed
Note that according to this definition {(used for instance in \cite{D,GGK,HKM})}, $\rs\,T\neq\emptyset$ implies that $T$ is a closed operator (note that this differs from the terminology used in \cite{K}).



\subsection{Bounded projections}

Let $\cX,\cY$ be subspaces of $\cH$ with $\cX\cap\cY=\{0\}$. The operator $P_{\cX,\cY}$ is bounded iff $\cX$, $\cY$ and $\cX+\cY$ are closed.

\begin{definition} Let $\Pro(\cH)$ denote the set of bounded projections on $\cH$. 
 Let $\Grass(\cH)$ denote the set of closed   subspaces of $\cH$ (the
 {\em Grassmannian} of $\cH$). 
\end{definition}

The first part of Prop \ref{pwo} can be adapted to the Banach space
setting as follows:
\begin{proposition}\label{pwo1}
Let  $J\in\cB_\linv({\red\cK'},\cH)$. Let $I\in\cB(\cH,\cK)$ be surjective. 
Clearly, $\Ran J,\Ker I$ are closed.
Then $IJ$ is invertible iff 
\[\Ran J\oplus \Ker I=\cH.\]
\end{proposition}
\proof Given Prop.
\ref{pwo}, it suffices to note that $IJ$ is a bijective bounded operator on a Banach space, hence it is invertible. \qeds

Here is an adaptation of the  first part of Prop. \ref{pq} to the
Banach space setting:

\begin{proposition} Let $\cX,\cX',\cY$ be closed subspaces of $\cH$, with $\cX$ complementary to $\cY$. Then  $\tilde P_{\cX,\cY}J_{\cX'} $ is invertible iff
$\cX'$ and $\cY$ are complementary.
\label{pro31}\end{proposition}

\subsection{Gap topology}

\bed If $\cX\in\Grass(\cH)$, we will introduce the following notation
for the  {\em ball} in $\cX$:
\[ B_\cX:=\{x\in\cX\ :\ \|x\|\leq1\}.\]
 As usual, the {\em distance} of a {non-empty} set $K\subset\cH$ and $x\in\cH$ is defined as
\[\dist(x,K):=\inf\{\|x-y\|\ :\ y\in K\}.\]
\eed

\begin{definition}
For $\cX,\cY\in\Grass(\cH)$ we define
\[\delta(\cX,\cY):=\sup\limits_{x\in B_\cX}\dist(x,\cY).\]
The {\em gap} between  $\cX$ and $\cY$ is defined as
\begin{equation}
\hat\delta(\cX,\cY):=\max\big(\delta(\cX,\cY),
\delta(\cY,\cX)\big).\label{gappo}\end{equation}
The {\em gap topology}  is the weakest topology on $\Grass(\cH)$ for which the function $\delta$ is continuous on
$\Grass(\cH)\times \Grass(\cH)$. 
\label{def-gap}
\end{definition}

Note that the gap defined in (\ref{gappo}) is not a metric. There exists a metric that can be used to define the gap topology, but we will not need it. We refer the reader to \cite{K} for more discussion about the gap topology.

\begin{proposition}\label{pri0}
If $\cX,\cY$ and $\cX',\cY'$ are two pairs of complementary subspaces in $\Grass(\cH)$, then
\[\max\big(\hat\delta(\cX,\cX'),\hat\delta(\cY,\cY')\big)\leq\|P_{\cX,\cY}-
P_{\cX',\cY'}\|.\]
\end{proposition}
\proof
For $x\in B_\cX$, we have
\begin{eqnarray*}
\dist(x,\cX')&\leq&\|x-P_{\cX',\cY'}x\|\\
&=&\|(P_{\cX,\cY}-P_{\cX',\cY'})x\|\ \leq\ \|P_{\cX,\cY}-P_{\cX',\cY'}\|.
\end{eqnarray*}This shows 
\[\delta(\cX,\cX')\leq\|P_{\cX,\cY}-
P_{\cX',\cY'}\|.\]
The same argument gives
\[\delta(\cX',\cX)\leq\|P_{\cX,\cY}-
P_{\cX',\cY'}\|.\]
Finally, we use 
\[P_{\cX,\cY}-P_{\cX',\cY'}=P_{\cY',\cX'}-P_{\cY,\cX}.\qed \]

\begin{corollary}\label{pri0.0}
 If $\cX,\cY$ and $\cX',\cY$ are complementary, then
\[\hat\delta(\cX,\cX')\leq\|P_{\cX,\cY}-
P_{\cX',\cY}\|.\]
\end{corollary}
\begin{lemma}
\[\|P_{\cY,\cX}P_{\cX',\cY'}\|\leq \|P_{\cY,\cX}\|\|P_{\cX',\cY'}\|\delta(\cX',\cX).\]
\end{lemma}
\proof
\begin{eqnarray*}
\|P_{\cY,\cX}P_{\cX',\cY'}\|&=&\sup_{v\in B_\cH}\|P_{\cY,\cX}P_{\cX',\cY'}v\|\\
&\leq&\|P_{\cX',\cY'}\|\sup_{x'\in B_{\cX'}}\|P_{\cY,\cX}x'\|\\
&=&\|P_{\cX',\cY'}\|\sup_{x'\in B_{\cX'}}\inf_{x\in\cX}\|P_{\cY,\cX}(x'-x)\|\\
&\leq&\|P_{\cY,\cX}\|\|P_{\cX',\cY'}\|\sup_{x'\in B_{\cX'}}\inf_{x\in\cX}\|x'-x\|.\qed
\end{eqnarray*}

The following proposition is essentially taken from \cite{GL}.

\begin{proposition}\label{pri1}
Let $\cX,\cY\in\Grass(\cH)$ be complementary, $\cX',\cY'\in\Grass(\cH)$ and {\begin{eqnarray}
\|P_{\cY,\cX}\|\delta(\cX',\cX)+
\|P_{\cX,\cY}\|\delta(\cY',\cY)&<&1,\label{contri}\\
\|P_{\cY,\cX}\|\delta(\cY,\cY')+
\|P_{\cX,\cY}\|\delta(\cX,\cX')&<&1.\label{contri0}
\end{eqnarray}}
Then $\cX',\cY'$ are complementary and
\begin{eqnarray}
\|P_{\cX,\cY}-P_{\cX',\cY'}\|&\leq&\frac{\|P_{\cX,\cY}\|\|P_{\cY,\cX}\|\big(\delta(\cX',\cX)+
\delta(\cY',\cY)\big)}
{1-\|P_{\cX,\cY}\|\delta(\cX',\cX)-\|P_{\cY,\cX}\|\delta(\cY',\cY)}.
\label{contri1}
\end{eqnarray}
\end{proposition}
\proof
Step 1. Let us show that $\cX'\cap\cY'=\{0\}$. Suppose it is not true. Then there exists $v\in \cX'\cap\cY'$, $\|v\|=1$. 
\begin{eqnarray*}
1=\|v\|&\leq&\|P_{\cX,\cY}v\|+\|P_{\cY,\cX}v\|\\
&\leq&\|P_{\cX,\cY}\|\delta(\cY',\cY)+\|P_{\cY,\cX}\|\delta(\cX',\cX),
\end{eqnarray*}
which is a contradiction.

Step 2.  By Step 1, $P_{\cX',\cY'}$ is well defined as a map on $\cX'+\cY'$. We will show that it is bounded, or equivalently that $\cX'+\cY'$ is closed. We will also obtain the estimate  (\ref{contri1}).

{
For simplicity, in the following estimates we assume that $\cH=\cX'+\cY'$.   If $\cX'+\cY'$ is strictly smaller than $\cH$, then we should replace
$P_{\cX,\cY}$ by $P_{\cX,\cY}J_{\cX'+\cY'}$.

Clearly,
\begin{eqnarray*}
P_{\cX,\cY}-P_{\cX',\cY'}&=&
P_{\cX,\cY}(P_{\cY',\cX'}+P_{\cX',\cY'})-(P_{\cX,\cY}+P_{\cY,\cX})P_{\cX',\cY'}\\
&=&
P_{\cX,\cY}P_{\cY',\cX'}-P_{\cY,\cX}P_{\cX',\cY'}.\end{eqnarray*}}
Hence,
\begin{eqnarray*}
\|P_{\cX,\cY}-P_{\cX',\cY'}\|&\leq&
\|P_{\cX,\cY}P_{\cY',\cX'}\|+\|P_{\cY,\cX}P_{\cX',\cY'}\|\\
&\leq&
\|P_{\cX,\cY}\|\|P_{\cY',\cX'}\|\delta(\cY',\cY)
+\|P_{\cY,\cX}\|\|P_{\cX',\cY'}\|\delta(\cX',\cX)\\
&\leq&
\|P_{\cX,\cY}\|\big(\|P_{\cY,\cX}\|+\|P_{\cX'\cY'}-P_{\cX,\cY}\|\big)\delta(\cY',\cY)\\
&&+
\|P_{\cY,\cX}\|\big(\|P_{\cX,\cY}\|+\|P_{\cX'\cY'}-P_{\cX,\cY}\|\big)\delta(\cX',\cX).
\end{eqnarray*}
Thus
\begin{eqnarray*}
&&\big(1-\|P_{\cX,\cY}\|\delta(\cX',\cX)-\|P_{\cY,\cX}\|\delta(\cY',\cY)\big)
\|P_{\cX,\cY}-P_{\cX',\cY}\|\\
&\leq&\|P_{\cX,\cY}\|\|P_{\cY,\cX}\|\big(\delta(\cX',\cX)+
\delta(\cY',\cY)\big).
\end{eqnarray*}

Step 3. We show that {$\cX'+\cY'=\cH$}. Suppose that this is not the case. We can then find  $v\in\cH$ such that $\|v\|=1$, {$\dist(v,\cX'+\cY')=1$.  }
Now 
\begin{eqnarray*}
1=\dist(v,\cX'+\cY')&\leq&\dist(P_{\cX,\cY}v,\cX'+\cY')+\dist(P_{\cY,\cX}v,\cX'+\cY')
\\
&\leq&\dist(P_{\cX,\cY}v,\cX')+\dist(P_{\cY,\cX}v,\cY')
\\&\leq &\|P_{\cX,\cY}\|\delta(\cX,\cX')+\|P_{\cY,\cX}\|\delta(\cY,\cY')
,\end{eqnarray*}
which is a contradiction.
\qed

\begin{corollary}
Let $\cX,\cY\in\Grass(\cH)$ be complementary and
{ \begin{equation}
\|P_{\cY,\cX}\|\delta(\cX',\cX)<1,\ \ \|P_{\cX,\cY}\|\delta(\cX,\cX')<1
.\label{contri.0}\end{equation}}
Then $\cX',\cY$ are complementary and
\begin{eqnarray}
\|P_{\cX,\cY}-P_{\cX',\cY}\|&\leq&\frac{\|P_{\cX,\cY}\|\|P_{\cY,\cX}\|\delta(\cX',\cX)}
{1-\|P_{\cX,\cY}\|\delta(\cX',\cX)}.\label{contri1.0}
\end{eqnarray}
\label{pri1.0}\end{corollary}

Note in passing that the proof of  Prop. \ref{pri1} shows also a somewhat more general statement (which we however will not use in the sequel): 
\begin{proposition}
{ If in Prop. \ref{pri1} we drop the condition (\ref{contri0}),
then $\cX'\cap\cY'=\{0\}$, $\cX'+\cY'$ is closed and the estimate 
(\ref{contri1}) is still true if we replace
$P_{\cX,\cY}$ with $P_{\cX,\cY}J_{\cX'+\cY'}$.}
\end{proposition}

\begin{proposition}\label{pro1}
Let  $T\in\cB_\linv(\cK,\cH)$, $S\in\cB(\cK,\cH)$. Then
\begin{equation}
\delta(\Ran T,\Ran S)\leq \|(T- S)\tilde T^{-1}\|.\label{psh}\end{equation}
Hence, if also  $S\in\cB_\linv(\cK,\cH)$,
\[\hat\delta(\Ran T,\Ran S)\leq \max\big(
\|(T- S)\tilde T^{-1}\|, \|(T- S)\tilde S^{-1}\|\big)
.\]
\end{proposition}
\proof Let $x\in B_{\Ran T}$. Clearly, $x=T\tilde T^{-1}x$. Hence
\begin{eqnarray*}
\dist(x,\Ran S)&\leq &\|x-S\tilde T^{-1}x\|\\
&=&\|(T-S)\tilde T^{-1}x\|\leq\|(T-S)\tilde T^{-1}\|,
\end{eqnarray*}
which proves (\ref{psh}). \qed

\begin{definition} A closed subspace of $\cH$ possessing a complementary subspace will be called {\em complemented}.
Let $\Grass_\com(\cH)$ stand for the family of complemented closed subspaces.

For any $\cY\in\Grass(\cH)$ let $\Grass_\cY(\cH)$ denote the family of closed subspaces complementary  to $\cY$.\end{definition}
The following fact  follows immediately from Prop.
\ref{pri1}:
\bep  $\Grass_\cY(\cH)$ and
 $\Grass_\com(\cH)$  are open subsets of
 $\Grass(\cH)$.\eep

\bep If $T\in \cB_\inv(\cH)$, then 
\[\Grass(\cH)\ni\cX\mapsto T\cX\in\Grass(\cH)\]
is bicontinuous. It preserves the complementarity relation, and hence it maps 
$\Grass_\com(\cH)$ into itself. \label{kjh1}\eep


\subsection{Continuous families of subspaces}

In this section $\Theta$ will be a locally compact space (eg. an open subset of $\C$.)
Consider a function 
\begin{equation}
\Theta\ni z\mapsto \cX_z\in \Grass(\cH).\label{func}\end{equation}

\begin{proposition} Let $\cY\in\Grass(\cH)$.
If  (\ref{func})  has values in $\Grass_\cY(\cH)$, then it 
is continuous  iff
\[  \Theta\ni z\mapsto P_{\cX_z,\cY}\in\Pr(\cH)\]
is continuous.
\end{proposition}
\proof We use Corrolaries \ref{pri0.0} and \ref{pri1.0}. \qed

\begin{definition}
 We say that 
\begin{equation}
\Theta\ni z\mapsto T_z\in \cB_\linv(\cK,\cH)
\label{func71}\end{equation} is an {\em injective  resolution} of (\ref{func}) if, for any
 $z\in\Theta$, $T_z$ is a bijection onto $\cX_z$.
\label{def-resol}
\end{definition}

\begin{proposition}\label{pro1a} Let $z_0\in\Theta$.
\begin{enumerate} \item 
If there exists an open $\Theta_0$ such that $z_0\in\Theta_0\subset\Theta$ and 
an injective  resolution of (\ref{func})
 on $\Theta_0$ continuous at $z_0$, then (\ref{func}) is continuous at $z_0$.
\item If  (\ref{func}) has values in $\Grass_\com(\cH)$, then we can put ``if and only if'' in 1.
\end{enumerate}
\end{proposition}
\proof (1): Suppose that (\ref{func71}) is an injective  resolution of
(\ref{func}) which is continuous at $z_0$. 
We can find an open  $\Theta_1$ such that $z_0\in\Theta_1^{\rm cl}\subset\Theta_0$ and for $z\in\Theta_1$ we have $\|T_z-T_{z_0}\|<c\|\tilde T_{z_0}^{-1}\|^{-1}$ with $c<1$. Then
$\|T_z\|$ and $\|\tilde T_z^{-1}\|$ are uniformly bounded for such  $z$. Therefore,  by
 Prop. \ref{pro1}, for such $z$ we have $\hat\delta(\cX_z,\cX_{z_0})\leq C\|T_z-T_{z_0}\|$. Thus the continuity of
 (\ref{func71}) implies the continuity of  (\ref{func}).

(2): Suppose that  (\ref{func}) is continuous at $z_0$ and 
 $\cX_{z_0}$ is complemented. Let
 $\cY\in\Grass(\cH)$ be complementary to $\cX_{z_0}$. There exists an
 open $\Theta_0$ such that $z_0\in\Theta_0\subset\Theta$ and $\cY$ is
 complementary to $\cX_z$, $z\in\Theta_0$. Then, by Prop. \ref{pro31},
 we see that
\[\Theta_0\ni z\mapsto P_{\cX_z,\cY}J_{\cX_{z_0}}\in \cB_\linv(\cX_{z_0},\cH)\]
is an injective  resolution of (\ref{func}) restricted to $\Theta_0$. \qeds

To sum up, functions with values in the Grassmannian that possess continuous injective  resolutions are continuous. If all values of a functions are  complemented, then the existence of a continuous injective  resolution can be adopted as an alternative definition of the continuity.

\subsection{Closed operators}

\begin{definition} For an operator $T$ on $\cH_1$ to $\cH_2$ with domain ${\rm Dom}T$ its graph is the subspace of $\cH_1\oplus\cH_2$ given by
\[{\rm Gr}(T):=\{(x,Tx)\in\cH_1\oplus\cH_2\ :\ x\in\Dom T\}.\]\end{definition}
This induces a map
\beq
\cC(\cH_1,\cH_2)\ni T\mapsto \Gr(T)\in \Grass(\cH_1\oplus\cH_2).\label{func0}
\eeq
From now on, we endow $\cC(\cH_1,\cH_2)$ with the gap topology transported from
$\Grass(\cH_1\oplus\cH_2)$ by (\ref{func0}).
\begin{proposition} 
$\cB(\cH_1,\cH_2)$ is  open in  $\cC(\cH_1,\cH_2)$. On $\cB(\cH_1,\cH_2)$ the gap topology coincides with the usual  norm
  topology.
 \end{proposition}

Operators whose graphs are complemented seem to have better properties.
\bed We denote by $\cC_\com(\cH_1,\cH_2)$ the set of closed operators with complemented graphs. \eed
Clearly,  $\cC_\com(\cH_1,\cH_2)$ is an open subset of  $\cC(\cH_1,\cH_2)$.

\bep\label{kjh2} Let $S\in\cB(\cH_1,\cH_2)$. Then the map
\beq \cC(\cH_1,\cH_2)\ni 
T\mapsto T+S\in \cC(\cH_1,\cH_2)\label{kjh}\eeq
is bicontinuous and preserves $\cC_\com(\cH_1,\cH_2)$
\eep
\proof (\ref{kjh}) on the level of graphs acts by
\[\left[\begin{array}{cc}\one&0\\S&\one\end{array}\right],\]
which is clearly in $\cB_\inv(\cH_1\oplus\cH_2)$. Thus the
proposition follows by Prop. \ref{kjh1}.
\qed

\bep \ben\item
Graphs of invertible operators are complemented.
\item
Graphs of operators whose resolvent set
  is nonempty  are complemented.\een
\eep
\proof By Prop. \ref{kjh2} applied to $-\lambda\one$, it is enough to
show (1).
We will show that if $T\in\cC_\inv(\cH_1,\cH_2)$, then
 $\cH_1\oplus\{0\}$ is complementary to $\Gr(T)$.

Indeed,
\[\cH_1\oplus\{0\}\cap\Gr(T)=\{0,0\}\]
is obviously true for any operator $T$.

Any $(v,w)\in\cH_1\oplus\cH_2$ can be written as
\[\big(v-T^{-1}w,0\big)+\big(T^{-1}w,w).\qed \]

\subsection{Continuous families of closed operators}

Consider a function
\begin{equation} \Theta\ni z\mapsto T_z\in\cC(\cH_1,\cH_2).\label{func1}\end{equation}

\begin{proposition} \label{propi}
Let $z_0\in\Theta$.
Suppose that there exists an open 
$\Theta_0$ such that $z_0\in\Theta_0\subset\Theta$, a Banach space $\cK$ and a  function \begin{equation}
\Theta_0\ni z\mapsto W_{z}\in \cB(\cK,\cH_1)\label{func2f}\end{equation}
 s.t. $W_z$ maps bijectively $\cK$ onto $\Dom(T_{z})$ for all $z\in\Theta_0$,
\begin{equation}
\Theta_0\ni z\mapsto T_{z}W_{z}\in \cB(\cK,\cH_2),\label{func3f}
\end{equation}
and both (\ref{func2f}) and (\ref{func3f}) are continuous at $z_0$. Then (\ref{func1}) is continuous at $z_0$.
\end{proposition}

\proof
Notice that
\beq\Theta_0\ni z\mapsto\big(W_z, T_zW_z)\in\cB_\linv(\cK,\cH_1\oplus\cH_2)\label{func2b}\eeq
 is an injective   resolution
 of
\begin{equation} \Theta_0\ni z\mapsto \Gr(T_z)\in\Grass(\cH_1\oplus\cH_2).\label{func2a}\end{equation}
(Actually,  every injective   resolution of (\ref{func2a})  is of the form (\ref{func2b}).)
The   injective resolution (\ref{func2b}) is continuous at $z_0$, hence
 (\ref{func1}) is continuous at $z_0$ by Prop. \ref{pro1a} (1). \qeds

{  A function $z\mapsto W_z$ with the properties described in 
Prop.
 \ref{propi} will be called a {\em resolution of continuity of $z\mapsto T_z$ at $z_0$}.}

{
\bep 
 Suppose that
 $z_0\in\Theta$,
 there exists an open 
$\Theta_0$ such that $z_0\in\Theta_0\subset\Theta$, and $T_z\in \cC_\inv(\cH_1,\cH_2)$, $z\in \Theta_0$.
Then
\[z\mapsto T_z\in\cC(\cH_1,\cH_2)\]
is continuous at $z_0$ iff
\[z\mapsto T_z^{-1}\in\cB(\cH_2,\cH_1)\]
is continuous at $z_0$.\label{priu}
\eep
\proof $\Rightarrow$: Let $z\mapsto W_z$ be a resolution of continuity of $z\mapsto T_z$ at $z_0$. Then
$z\mapsto V_z:=T_zW_z$
is invertible bounded and continuous at $z_0$. Hence, so is $z\mapsto V_z^{-1}$. Therefore, $z\mapsto T_z^{-1}=W_zV_z^{-1}$ is bounded and continuous at $z_0$.

$\Leftarrow$: Obviously, $z\mapsto T_z^{-1}$ is a resolution of continuity of $z\mapsto T_z$ at $z_0$. \qeds

The following proposition is an immediate consequence of Prop. \ref{priu}.}
\bep { Let $\lambda\in\C$ and consider a function
\beq\label{eq:deftzt}
z\mapsto T_z\in\cC(\cH).
\eeq}{
Suppose there exists an open $\Theta_0$ such that $z_0\in\Theta_0\subset\Theta$, and $\lambda\in\rs (T_z)$, $z\in \Theta_0$.
Then (\ref{eq:deftzt}) is continuous at $z_0$ iff
\[z\mapsto (\lambda\one-T_z)^{-1}\in\cB(\cH)\]
is continuous at $z_0$.\label{priu1}}
\eep

\subsection{Holomorphic families of closed subspaces}

Let $\Theta$ be an open subset of ${\mathbb C}$ and suppose we are given a function
\begin{equation}
\Theta\ni z\mapsto\cX_z\in\Grass(\cH).\label{qeq6}\end{equation}

\begin{definition}
We will say that the family (\ref{qeq6}) is {\em complex differentiable at $z_0$} if
there exists an open $\Theta_0$ with $z_0\in\Theta_0\subset\Theta$ and an injective resolution
\begin{equation}
\Theta_0\ni z\mapsto T_z\in B_\linv(\cK,\cH)\label{qeq7}\end{equation}
 of (\ref{qeq6}) complex differentiable at $z_0$.
If  (\ref{qeq6})  is complex differentiable on the whole $\Theta$, we say it is {\em holomorphic}.
\label{def-holo}
\end{definition}

Clearly, the complex differentiability implies the continuity.

\bep Suppose that 
 (\ref{qeq6}) is  complex differentiable at  $z_0\in \Theta$.
Let $\Theta_0$ be open with $z_0\in\Theta_0\subset\Theta$, and let
$\cY$ be a subspace complementary to
 $\cX_z$, $z\in\Theta_0$. Then the family 
 $\Theta_0\ni z\mapsto P_{\cX_z,\cY}$ is complex differentiable at $z_0$.
\label{stw1}\eep

\proof By making, if needed, $\Theta_0$ smaller, we can assume that we
have an injective  resolution $\Theta_0\ni z\mapsto T_z$ complex
differentiable at $z_0$.  For such
$z$, by Prop. \ref{pwo1},
 $\tilde P_{\cX_{z_0},\cY}T_z$ is invertible. Therefore, by Prop. \ref{pwo},
\beq
P_{\cX_z,\cY}= T_z(\tilde P_{\cX_{z_0},\cY} T_z)^{-1}\tilde
P_{\cX_{z_0},\cY}.\label{pwi}\eeq 
(\ref{pwi}) is clearly complex differentiable at $z_0$. \qed

\bep\label{stw6}
Let $z_0\in \Theta$. Suppose that $\cX_{z_0}\in\Grass_\com(\cH)$. 
The following are equivalent:
\begin{enumerate}
\item
 (\ref{qeq6}) is  complex differentiable at $z_0$.
\item
There exists an open $\Theta_0$ with $z_0\in\Theta_0\subset\Theta$ and a closed subspace $\cY$ complementary to
 $\cX_z$, $z\in\Theta_0$, such that the family 
 $\Theta_0\ni z\mapsto P_{\cX_z,\cY}$ is complex differentiable at $z_0$.
\end{enumerate}
\eep

\proof (1)$\Rightarrow$(2): Consider the injective   resolution
(\ref{qeq7}). Let $\cY$ be a subspace complementary to $\cX_{z_0}=\Ran
T_{z_0}$. We know that $\Theta\ni z\mapsto\cX_z$ is continuous by
Prop. \ref{pro1a}. Hence, by taking $\Theta_0$ smaller, we can assume that
$\cY$ is complementary to $\cX_z$, $z\in\Theta_0$.
By Prop. \ref{stw1}, $P_{\cX_z,\cY}$ is complex differentiable at $z_0$.




 (2)$\Rightarrow$(1):
  \[\Theta_0\ni z\mapsto P_{\cX_z,\cY}J_{\cX_{z_0}}\in B_\linv(\cX_{z_0},\cH)\] is an injective resolution of  (\ref{qeq6}) complex differentiable at $z_0$. \qed




\bep[Uniqueness of analytic continuation for subspaces]\label{theorem:bruk1} Let $\Theta\subset\C$ be connected
  and open. Let 
\[\Theta\ni z\mapsto \cX_{z},\cY_{z}\in
  \Grass_\com (\cH)\] be holomorphic. 
Consider a sequence $\left\{z_1,z_2,\dots\right\}\subset\Theta$  converging to a point $z_0\in\Theta$ s.t. $z_n\neq z_0$ for each $n$. Suppose $\cX_{z_n}=\cY_{z_n}$, $n=1,2,\dots$. Then $\cX_{z}=\cY_{z}$ for all $z\in\Theta$.
\eep
\begin{proof} For  holomorphic functions with values in bounded operators the unique continuation property is straightforward. Therefore, it suffices to apply Prop. \ref{stw6}.\end{proof}

\bep Let (\ref{qeq6}) and
 $\Theta\ni z\mapsto T_z\in \cB(\cH,\cH_1)$ be holomorphic. Suppose
that for all $z\in\Theta$, $T_z$ is injective on $\cX_z$ and
$T_z\cX_z$ is closed. Then
\beq
\Theta\ni z\mapsto T_z\cX_z\in \Grass(\cH_1)\label{stwi6}\eeq
is holomorphic.\label{stwi9}\eep
\proof
Let $\Theta_0\ni z\mapsto S_z\in \cB_\linv(\cK,\cH)$ be a holomorphic injective
 resolution of  (\ref{qeq6}). Then $ \Theta_0\ni z\mapsto T_zS_z\in
\cB_\linv(\cK,\cH_1)$ is a holomorphic injective   resolution of
(\ref{stwi6}). \qed

\subsection{Holomorphic families of closed operators}

Consider a function
\begin{equation}
\Theta\ni z \mapsto T_z\in \cC(\cH_1,\cH_2).\label{qeq5a}\end{equation}
\begin{definition} Let $z_0\in\Theta$. We say that 
 (\ref{qeq5a}) is complex differentiable  at $z_0$ if 
\begin{equation}
\Theta\ni z \mapsto \Gr(T_z)\in \Grass(\cH_1\oplus\cH_2)\label{qeq5c}\end{equation}
is complex differentiable at $z_0$.
\label{def-holo-op}
\end{definition}

{ The following proposition gives an equivalent condition, which in most of the literature is adopted as the basic  definition of the  complex differentiablity of functions with values in closed operators.}

\begin{proposition}
 { (\ref{qeq5a}) is complex differentiable  at  $z_0\in\Theta$
iff there exists an open 
$\Theta_0$ such that $z_0\in\Theta_0\subset\Theta$, a Banach space $\cK$ and a  function \begin{equation}
\Theta_0\ni z\mapsto W_{z}\in \cB(\cK,\cH_1)\label{func2}\end{equation}
 s.t. $W_z$ maps bijectively $\cK$ onto $\Dom(T_{z})$ for all $z\in\Theta_0$, 
\begin{equation}
\Theta_0\ni z\mapsto T_{z}W_{z}\in \cB(\cK,\cH_2),\label{func3}
\end{equation} and both
 (\ref{func2}) and (\ref{func3}) are 
 are complex differentiable at $z_0$.
\label{func7}}
\end{proposition}

\proof We use the fact, noted in the proof of Prop. \ref{propi}, that
(\ref{func2b})
 is an injective   resolution of
(\ref{func2a}), { and that every injective resolution is of this form}. \qeds

{  A function $z\mapsto W_z$ with the properties described in 
Prop.
 \ref{func7} will be called a {\em  resolution of complex differentiability of $z\mapsto T_z$ at $z_0$}.}

The following theorem follows immediately from Thm \ref{theorem:bruk1}:

\begin{theorem}[Uniqueness of analytic continuation for closed operators \cite{B}]\label{theorem:bruk} Let $\Theta\subset\C$ be connected
  and open. Let \[\Theta\ni z\mapsto T_{z},S_{z}\in
  \cC_\com(\cH_1,\cH_2)\] be holomorphic. 
Consider a sequence $\left\{z_1,z_2,\dots\right\}\subset\Theta$  converging to a point $z_0\in\Theta$ s.t. $z_n\neq z_0$ for each $n$. Suppose $T_{z_n}=S_{z_n}$, $n=1,2,\dots$. Then $T_{z}=S_{z}$ for all $z\in\Theta$.
\end{theorem}

We also have the holomorphic obvious analogs of Props \ref{priu} and \ref{priu1}, with the word ``continuous'' replaced by ``complex differentiable''.

\subsection{Holomorphic families in the dual space}

\bed Let $\cH^*$ denote the {\em dual space} of $\cH$. We adopt the convention that $\cH^*$ is the space  of anti-linear continuous functionals, cf. \cite{K}. (Sometimes $\cH^*$ is then called the {\em antidual space}). If $\cX\in\Grass(\cH)$, we denote by $\cX^\bot\in\Grass(\cH^*)$ its {\em annihilator}. If $ T\in\cC(\cK,\cH)$ is densely defined, then $T^*\in\cC(\cH^*,\cK^*)$ denotes its {\em adjoint}. \eed

Let $\cX,\cY\in\Grass(\cH)$ be two complementary subspaces. Then $\cX^\bot$, $\cY^\bot$ are also complementary and
\[P_{\cX,\cY}^*=P_{\cY^\bot,\cX^\bot}.\]


{In the proof of the next theorem we will  use the equivalence of various definitions of the holomorphy of functions with values in bounded operators mentioned at the beginning of the introduction \cite{K}.}

\bep[Schwarz reflection principle for subspaces]  A function \[z\mapsto \cX_{z}\in\Grass_\com(\cH)\label{funcinit}\] 
is complex differentiable at $z_0$ {iff}
\beq  z\mapsto \cX_{\bar z}^\perp\in\Grass_\com(\cH^*)\label{func5}\eeq
is complex differentiable at $\bar z_0$. \label{stwi7}\eep

\proof Locally, we can choose $\cY\in\Grass(\cH)$ complementary to
$\cX_z$. If (\ref{funcinit}) is holomorphic then $z\mapsto P_{\cX_z,\cY}$ is holomorphic by Prop. \ref{stw1}. But $\cY^\perp$ is complementary to $\cX_z^\perp$ and
\beq\label{eq:relcompl}
P_{\cX_{z}^\perp,\cY^\perp}=\one-P_{\cX_z,\cY}^*.
\eeq
So $z\mapsto P_{\cX_{\bar z}^\perp,\cY^\perp}$ is complex differentiable at $\bar z_0$. This
means that (\ref{func5}) is complex differentiable.

{Conversely, if (\ref{func5}) is complex differentiable at $\bar z_0$ then $z\mapsto P_{\cX_{{\bar z}}^\perp,\cY^\perp}$ is complex differentiable. By (\ref{eq:relcompl}) this implies $z\mapsto P_{\cX_{\bar z},\cY}^*$ is complex differentiable at ${\bar z}_0$. Therefore, $z\mapsto \bra u| P_{\cX_{\bar z},\cY}^* v \ket$ is complex differentiable for all $u\in \cH^{**}$, $v\in\cH^*$. In particular, by the embedding $\cH\subset\cH^{**}$,  $z\mapsto \bra  u |P_{\cX_{\bar z},\cY}^* v \ket =\overline{\bra v | P_{\cX_{\bar z},\cY} u  \ket}$ is holomorphic for all $u\in \cH$, $v\in\cH^*$. This proves $P_{\cX_z,\cY}$ is complex differentiable at $z_0$, thus (\ref{funcinit}) is complex differentiable as claimed.
\qed

\begin{remark}A direct analogue of Prop. \ref{stwi7} holds for continuity, as can be easily shown using the identity $\delta(\cX,\cY)=\delta(\cY^\perp,\cX^\perp)$ for closed subspaces $\cX,\cY\subset\cH$, cf. \cite{K}.\end{remark}}

We have  an analogous property for functions with values in closed operators.

\begin{theorem}[Schwarz reflection principle for closed operators] \label{cor:adjointhol}Let 
\beq\label{eq:deftsr} 
z\mapsto T_z\in\cC_\com(\cH_1,\cH_2)
\eeq
have values in densely defined operators. Then it is complex differentiable at $z_0$ {iff} \[z\mapsto T^*_{\bar z}\in\cC_\com(\cH_2^*,\cH_1^*)\]
 is complex differentiable at $\bar{z}_0$.
\end{theorem}
\proof
It is well known that $\Gr(T_z)^\perp= U\Gr(T_z^*)$, where $U$ is the invertible operator given by $U(x,y)=(-y,x)$ for $(x,y)\in\cH_1^*\oplus\cH_2^*$. Thus, the {equivalence of the holomorphy of $z\mapsto \Gr(T_{\bar z}^*)$ and $z\mapsto \Gr(T_{ z})$} follows from Prop. \ref{stwi7} and  Prop. \ref{stwi9} { applied to the constant bounded invertible operator} {$U$ or $U^{-1}$}.\qeds

 We will make use of the following well-known result:

\begin{theorem}[{\cite[Thm. 5.13]{K}}]\label{thm:ar} Let $T\in\cC(\cH_1,\cH_2)$ be densely defined. Then $\Ran T$ is closed iff $\Ran T^*$ is closed in $\cH_1^*$. In such case,
\beq\label{eq:perpran}
(\Ran T)^{\bot}=\ker \,T^*, \quad (\ker\, T)^\bot=\Ran T^*.
\eeq
\end{theorem}



\bep\label{stwi8}
 Let $\Theta \ni z\mapsto S_z\in \cC_\com(\cH_1,\cH_2)$ be holomorphic. Assume that ${\rm Dom}(S_z)$ is dense and $\Ran S_z=\cH_2$. Then
\beq \Theta\ni z\mapsto \Ker S_{ z}\in\Grass(\cH_1)\label{stwi3}\eeq
is holomorphic.
\eep
\proof
{ 
Since $z\mapsto S_z$ is holomorphic, so is $z\mapsto S_{\bar z}^*$. Let $z\mapsto W_z$  be a resolution of holomorphy of $z\mapsto S_{\bar z}^*$. 
 By (\ref{eq:perpran}), $\Ker S_{\bar z}^*=(\Ran S_{\bar z})^\perp=\{0\}$. It follows that $z\mapsto S_{\bar z}^*W_z$ is a holomorphic injective resolution of 
\beq
z\mapsto \Ran S_{\bar z}^*\in\Grass(\cH_1^*).\label{eq:Sadj}
\eeq
Hence (\ref{eq:Sadj}) is holomorphic. But $(\Ker S_z)^\perp=\Ran S_{\bar z}^*$.
Hence $z\mapsto \Ker S_z$ is holomorphic by the Schwarz reflection principle for subspaces (Prop. \ref{stwi7}).
\qed}

\section{Hilbert space theory}\label{sec4}

Throughout this section 
 $\cH,\cH_1,\cH_2$ are Hilbert spaces.
Note that $\cH^*$ can be identified with $\cH$ itself, the annihilator can be identified with the orthogonal complement and the adjoint with the Hermitian adjoint.

\subsection{Projectors}
\begin{definition}
We will use the term {\em projector} as the synonym for {\em orthogonal projection}. We will write $P_\cX$ for the projector onto $\cX$.
\end{definition}
 Thus, $P_\cX=P_{\cX,\cX^\perp}$. 

Let $\cX,\cY$ be subspaces of $\cH$.
Then $\cX\oplus\cY^\perp=\cH$ is equivalent to  $\cX^\perp\oplus\cY=\cH$, which  is equivalent to 
  $\|P_\cX-P_\cY\|<1$.

The gap topology on the Grassmannian of a Hilbert space simplifies
considerably. In particular, the gap function is a metric and has a
convenient expression in terms of the projectors:
\[\hat\delta(\cX_1,\cX_2):=\|P_{\cX_1}-P_{\cX_2}\|,\ \ \cX_1,\cX_2\in\Grass(\cH).\]

Thus a function
\beq \Theta\ni z\mapsto \cX_z\in\Grass(\cH)\label{func4}\eeq
is continuous iff $z\mapsto
P_{\cX_z}$ is continuous. Unfortunately, the analogous statement is
not true for the holomorphy, and we have to use the criteria discussed
in the section on Banach spaces. This is however simplified by the fact that in a Hilbert space each closed subspace is complemented, so that $\Grass(\cH)=\Grass_\com(\cH)$ and $\cC(\cH_1,\cH_2)=\cC_\com(\cH_1,\cH_2)$.

\subsection{Characteristic matrix}

\begin{definition}
The {\em characteristic matrix} of a closed operator $T\in\cC(\cH_1,\cH_2)$  is defined as the projector onto $\Gr(T)$ and denoted $M_T$.\end{definition}
Assume that $T$ is densely defined. We set
\[\langle T\rangle:=(\one +T^*T)^{\frac12}.\]
It is easy to see that
\[J_T:=
\left[\begin{array}{c}\langle T\rangle^{-1}\\
T\langle T\rangle^{-1}\end{array}\right]:\cH_1\to\cH_1\oplus\cH_2\]
 is a partial isometry onto $\Gr(T)$. Therefore, by (\ref{pq2}) $M_T=J_TJ_T^*.$
To obtain a more explicit formula for the characteristic matrix, note the identities
\begin{eqnarray*}
T(\one +T^*T)^{-1}&=&\big((\one +TT^*)^{-1}T\big)^{\rm cl},\\
TT^*(\one +TT^*)^{-1}&=&\big(T(\one +T^*T)^{-1}T^*\big)^{\rm
  cl}\ =\ \big((\one +TT^*)^{-1}TT^*\big)^{\rm cl}.
\end{eqnarray*}
Note that the above formulas involve products of unbounded operators.
We use the standard definition of the product of unbounded
operators recalled in Def. \ref{prodi}.

In the following formula for the
 characteristic matrix we are less pedantic and we omit the superscript
 denoting the closure: 
\begin{equation}
M_T:=\left[\begin{array}{cc}\langle T\rangle^{-2}&\langle T\rangle^{-2}T^*\\
T\langle T\rangle^{-2}&T\langle T\rangle^{-2}T^*\end{array}\right],
\label{qeq2}\end{equation}

Consider a function 
\begin{equation}
\Theta\ni z\mapsto T_z\in\cC(\cH_1,\cH_1).\label{func-}\end{equation}

\begin{proposition}
Let $z_0\in\Theta$. The function 
(\ref{func-}) is continuous in the gap topology at $z_0$ iff the functions
\begin{eqnarray}
\Theta\ni z&\mapsto &  \langle T_z\rangle^{-2}\in\cB(\cH_1,\cH_1),\label{kas1}\\
\Theta\ni z&\mapsto &T_z\langle T_z\rangle^{-2}\in\cB(\cH_1,\cH_2)\label{kas2},\\
\Theta\ni z&\mapsto &  \langle T_z^*\rangle^{-2}\in\cB(\cH_2,\cH_2)\label{kas3}
\end{eqnarray}
are continuous at $z_0$. \end{proposition}

\proof Clearly,  (\ref{func1}) is continuous in the gap topology iff
$\Theta\ni z\mapsto M_{T_z}$ is. 
Now
(\ref{kas1}), (\ref{kas2}) resp. (\ref{kas3})
are $\big(M_{T_z}\big)_{11}$,
 $\big(M_{T_z}\big)_{12}=\big(M_{T_z}\big)_{21}^*$, resp.
 $\one-\big(M_{T_z}\big)_{22}$. \qeds

The gap topology is not the only topology on $\cC(\cH_1,\cH_2)$ that on $\cB(\cH_1,\cH_2)$ coincides with the usual norm topology.  Here is one of the examples considered in the literature:

\begin{definition}
We say that 
(\ref{func1}) is continuous at $z_0$  in the {\em  Riesz topology}
\footnote{Note that this name is not used in older references.}
 if \[\Theta\ni z\mapsto T_{z}\langle T_{z}\rangle^{-1}\in\cB(\cH_1,\cH_2)\]
is continuous at $z_0$.
\end{definition}

It is easy to see that
 the Riesz topology is strictly stronger than the gap topology
 \cite{kaufman}. Indeed, the fact that it is stronger is obvious. Its non-equivalence with the gap topology on self-adjoint operators follows from the following easy fact \cite{nicolaescu}:

\begin{theorem} Suppose that the values of (\ref{func1}) are self-adjoint.
\begin{enumerate}\item
 The following are equivalent:
\begin{romanenumerate}
\item $\lim\limits_{z\to z_0}T_z= T_{z_0}$ in the gap topology;
\item $\lim\limits_{z\to z_0}f(T_z)= f(T_{z_0})$ for all bounded continuous $f:\R\to\C$ such that $\lim\limits_{t\to-\infty}f(t)$ and $\lim\limits_{t\to+\infty}f(t)$ exist and are equal.
\end{romanenumerate}
\item The following are equivalent:
\begin{romanenumerate}
\item $\lim\limits_{z\to z_0}T_z= T_{z_0}$ in the Riesz topology;
\item $\lim\limits_{z\to z_0}f(T_z)= f(T_{z_0})$ for all bounded continuous $f:\R\to\C$ such that $\lim\limits_{t\to-\infty}f(t)$ and $\lim\limits_{t\to+\infty}f(t)$ exist.
\end{romanenumerate}
\end{enumerate}
\end{theorem}

\subsection{Relative characteristic matrix}

Let $T$ and $S$ be densely defined closed operators. 
\begin{theorem}\label{qeq4} We have
\[\Gr(T)\oplus\Gr(S)^\perp=\cH_1\oplus\cH_2\] iff $J_S^*J_T$
is invertible. Then the projection onto $\Gr (T)$ along $\Gr (S)^\perp$ is  given by
\begin{equation}\label{eq:chmax}
J_T(J_S^*J_T)^{-1}J_S^*\end{equation}
\end{theorem}

\proof We have $J_T\in \cB_\inv(\cH_1,\cH_1\oplus\cH_2)$. $J_S^*\in\cB(\cH_1\oplus\cH_2,\cH_1)$ is surjective. Besides, $\Ran J_T=\Gr(T)$ and $\Ker J_S^*=\Gr(S)^\perp$. Therefore, it suffices to apply first Prop. \ref{pwo1}, and then Prop. \ref{pwo}. \qed

\begin{definition} (\ref{eq:chmax}) will be called the {\em relative characteristic matrix of $T,S$} and will be denoted $M_{T,S}$.
\end{definition}

Clearly, $M_T=M_{T,T}$.

We can formally write
\begin{equation}
J_S^*J_T=\langle S\rangle^{-1}(\one +S^*T)\langle T\rangle^{-1}.
\label{qeq5}\end{equation}
To make  (\ref{qeq5}) rigorous we interpret
 $(\one +S^*T)$  as a bounded operator from 
$\langle T\rangle^{-1}\cH_1$ to $\langle S \rangle\cH_1$.
Now the inverse of  (\ref{qeq5}) is a bounded operator, which can be formally written as
\[(J_S^*J_T)^{-1}=\langle T \rangle(\one +S^*T)^{-1}\langle S\rangle.\]

Thus we can write
\begin{eqnarray}\nonumber M_{T,S}&=&
\left[\begin{array}{c}\langle T\rangle^{-1}\\
T\langle T\rangle^{-1}\end{array}\right]
\langle T \rangle(\one +S^*T)^{-1}\langle S\rangle
\left[\begin{array}{cc}\langle S\rangle^{-1}&\langle S\rangle^{-1}S^*
\end{array}\right]\\[3ex]
&=&
\left[\begin{array}{cc}(\one +S^*T)^{-1}&(\one +S^*T)^{-1}S^*\\
T(\one +S^*T)^{-1}&T(\one +S^*T)^{-1}S^*\end{array}\right]\label{qe3}\\
&=&
\left[\begin{array}{cc}(\one +S^*T)^{-1}&S^*(\one +TS^*)^{-1}\\
T(\one +S^*T)^{-1}&TS^*(\one +TS^*)^{-1}\end{array}\right].
\label{qeq3a}\end{eqnarray}
Note that even though the entries of (\ref{qe3}) and (\ref{qeq3a}) are expressed in terms of unbounded operators, all of them can be interpreted as bounded everywhere defined operators (eg. by taking the closure of the corresponding expression).

\subsection{Holomorphic families of closed operators}

In order to check the holomorphy of a function 
\begin{equation} \Theta\ni z\mapsto
  T_z\in\cC(\cH_1,\cH_2)\label{func1f}\end{equation} using
 the criterion given in Prop. \ref{func7} one needs to find a
 relatively arbitrary function $z\mapsto W_z$. 
In the case of a Hilbert space we have a criterion for the complex
differentiability involving  relative  characteristic
matrices. We believe that this criterion should be often more
convenient, since it involves a function with values in bounded
operators uniquely  defined for any $z_0\in\Theta$.

\begin{proposition}
Let $z_0\in\Theta$ and assume (\ref{func1f}) has values in densely defined operators. Then (\ref{func1f}) is complex differentiable  at $z_0$ if there exists an open 
$\Theta_0$ such that $z_0\in\Theta_0\subset\Theta$, and for $z\in\Theta_0$,
\[\langle T_{z_0}\rangle^{-1}(\one+T_{z_0}^*T_z)\langle T_z\rangle^{-1}\]
is invertible, so that we can define
\beq
\left[\begin{array}{cc}(\one+T_{z_0}^*T_z)^{-1}&T_{z_0}^*(\one+T_z T_{z_0}^*)^{-1}\\
T_z(\one+T_{z_0}^*T_z)^{-1}&T_z T_{z_0}^*(\one+T_z T_{z_0}^*)^{-1}\end{array}\right]\in\cB(\cH_1\oplus\cH_2),
\label{qeq3d}\eeq
and (\ref{qeq3d}) is complex differentiable at $z_0$.
\label{qeq9}\end{proposition}


\section{Products and sums of operator--valued holomorphic functions}\label{sec5}

In this section we focus on the question what conditions
ensure that the product and the sum of two holomorphic families of closed
operators is
  holomorphic. Note that analogous statements hold for
 families continuous in the gap topology.

Throughout the section
 $\cH_1,\cH_2,\cH$ are Banach spaces.
  
\subsection{Products of closed operators I}\label{ss:products}

 If both { $A,B$ are  closed operators, then the
 product $AB$ (see Def. \ref{prodi})
 does not need to be closed.} 
 We recall below standard criteria for this to be true.
For a more detailed discussion and other sufficient
 conditions we refer the reader to \cite{azizov} and references
 therein. 

\begin{proposition}\label{prop:criterionAB1} Let
\begin{enumerate}\item $A\in\cC(\cH,\cH_2)$ and $B\in\cB(\cH_1,\cH)$, or
\item $A\in\cC_\inv(\cH,\cH_2)$ and $B\in\cC(\cH_1,\cH)$.
\end{enumerate}
Then $AB$ is closed.
\end{proposition}

The simpliest conditions which imply holomorphy of the product are listed in the proposition below. Unconveniently, they are not quite compatible with the sufficient conditions for the closedness of the product, which has to be assumed separately.

\bep\label{criterion1}  Let
\begin{enumerate}\item $\Theta\ni z\mapsto A_z\in \cB(\cH,\cH_2), B_z \in \cC(\cH_1,\cH)$ be holomorphic, or
 \item $\Theta\ni z\mapsto A_z\in \cC(\cH,\cH_2), B_z \in \cC_{\inv}(\cH_1,\cH)$ be holomorphic.\end{enumerate}
If in addition $A_z B_z\in\cC(\cH_1,\cH_2)$ for all $z\in\Theta$ then $\Theta\ni z\mapsto A_z B_z$ is holomorphic.
\eep
\proof { (1) Let
$z\mapsto V_z$ be a resolution of holomorphy of $z\mapsto B_z$. Then it is also a resolution of holomorphy of $z\mapsto A_zB_z$.

(2) Let $z\mapsto U_z$ be a resolution of holomorphy of $z\mapsto A_z$. By the holomorphic version of Prop. \ref{priu}, $z\mapsto B_z^{-1}\in \cB(\cH,\cH_1)$ is holomorphic. It is obviously injective. Hence
  $B_z^{-1}U_z$ is a resolution of holomorphy of  $z\mapsto A_z B_z$. }
\qed

\subsection{Examples and counterexamples}\label{ss:counterex}
\bed The {\em point spectrum} of $T$ is defined as
\[\sp_\p(T):=
\{z\in\C\ :\ \Ker(A-z\one)\neq\{0\}\}.\]
\eed

\begin{example} 
Let  $T\in \cB(\cH)$. Then $z\mapsto(z\one-T)^{-1}$ is holomorphic on $\C\backslash\sp_\p(T)$. Indeed,  $z \mapsto W_z:=z\one-T$ is injective, holomorphic, $\Ran W_z=\Dom(z\one-T)^{-1}$ and $(z\one-T)^{-1}W_z=\one$.
\label{exa}\end{example}

\begin{example} The above example can be generalized.
Let  $T\in \cC(\cH)$ have a nonempty resolvent set. Then $z\mapsto(z\one-T)^{-1}$ is holomorphic on $\C\backslash\sp_\p(T)$. Indeed, let $z_0\in\rs(T)$. Then $z \mapsto W_z:=(z\one-T)(z_0\one-T)^{-1}$ is injective, holomorphic, $\Ran W_z=\Dom(z\one-T)^{-1}$ and $(z\one-T)^{-1}W_z=(z_0\one-T)^{-1}$.
\label{exa+}\end{example}

\begin{example}
Consider $A_z:=T$ with $T\in\cC(\cH)$ unbounded and $B_z:=z\one\in\cB(\cH)$. 
Then the product $A_z B_z$ is closed for all $z\in\C$, but the function $z\mapsto A_z B_z$ is not complex differentiable at $z=0$ due to the fact that that it yields a bounded operator at $z=0$, but fails to do so in any small neighbourhood (cf. Example 2.1 in \cite[Ch. VII.2]{K}).
\end{example}

Therefore, it is not true that if $A_z$ and $B_z$ are holomorphic and $A_z B_z$ is closed for all $z$, then  $z\mapsto A_z B_z$ is
holomorphic.

The more surprising fact is that even the additional requirement that $A_z B_z$ is bounded does not guarantee the holomorphy, as shows the example below.

\begin{example} Assume that $T\in\cC(\cH)$ has empty spectrum. Note that
this implies that $\sp(T^{-1})=\{0\}$ and $\sp_\p(T^{-1})=\emptyset$.
By Example \ref{exa}, $A_z:=T(Tz-\one)^{-1}=(z\one-T^{-1})^{-1}$ is holomorphic. Obviously, so is $B_z:= z\one $. Moreover, $z\mapsto A_z B_z=\one+(Tz-\one)^{-1}$ has values in bounded operators. However, it is not differentiable at zero, because \[\partial_zA_zB_zv\Big|_{z=0}={\red -}Tv,
\ \ \  v\in\Dom T,\]
 and $T$ is unbounded.
\end{example}




\subsection{Products of closed operators II}\label{ss:products2}

In this subsection, exceptionally, $\cH_1,\cH_2,\cH$ are Hilbert spaces.

We quote below a useful criterion specific to that case. The proof is not difficult and can be found for instance in \cite[Prop 2.35]{DG}.

\begin{proposition}\label{prop:criterionAB2}Let
 $A\in \cB(\cH,\cH_2)$, $B\in \cC(\cH_1,\cH)$. If {\red $\Dom( B)$ and } $\Dom(B^*A^*)$ {\red are}
  dense, then $AB$ is closable, $B^*A^*$ is closed and $(AB)^*=B^*
  A^*$.  
\end{proposition}

Together with Prop. \ref{criterion1}, this yields the following result.	

\begin{proposition}\label{criterion2}
Let $\Theta\ni z\mapsto A_z\in \cC(\cH,\cH_2)$, $\Theta\ni z\mapsto
B_z \in \cB(\cH_1,\cH)$ be holomorphic. If $A_z$, $A_z B_z$ are densely
defined and $B_z^* A_z^*$ is closed for all $z\in\Theta$, then
$\Theta\ni z\mapsto A_z B_z\in\cC(\cH_1,\cH_2)$ is holomorphic.  
\end{proposition}
\begin{proof} {Since $A_z$ is closed and densely defined, $A_z^{**}=A_z$. By  Prop. \ref{prop:criterionAB1}, $A_z B_z$ is
  closed. Thus, we can apply Prop. \ref{prop:criterionAB2} to $A:=B_z^*$ and $B:=A_z^*$ for all $z\in\Theta$ and conclude
  \beq\label{eq:azbzadj}
  A_z B_z = A_z^{**} B_z^{**}=(B_z^* A_z^*)^*
  \eeq
By the Schwarz reflection principle and Prop. \ref{criterion1}, $z\mapsto B_{\bar z}^* A_{\bar z}^*$ is holomorphic. Therefore, $z\mapsto A_z B_z$ is holomorphic by (\ref{eq:azbzadj}) and the Schwarz reflection principle.}\end{proof}

\subsection{Non-empty resolvent set case}

Here is another sufficient condition for the holomorphy of the product of operator--valued holomorphic functions, based on a different strategy.

In the statement of our theorem below, the closedness of the products $A_z B_z$ and $B_z A_z$ is implicitly assumed in the non-empty resolvent set condition.

\begin{theorem}\label{theorem:holoproduct}
Let $\Theta\ni z\mapsto A_z\in\cC(\cH_2,\cH_1)$ and $\Theta\ni z\mapsto B_z\in\cC(\cH_1,\cH_2)$ be holomorphic.
 Assume that there exists $\lambda\in\C$ s.t. $\lambda^2\in{\rm rs}(A_z B_z)\cap{\rm rs}(B_z A_z)$ for all $z\in\Theta$. Then both $\Theta\ni z \mapsto A_z B_z, B_z A_z$ are holomorphic.
\end{theorem}
 
The proof is based on the helpful trick  of replacing the study of the product $A_z B_z$ by the investigation of the operator $T_z$ on $\cH_1\oplus\cH_2$ defined by
\begin{equation}\label{eq:defT}
T_z:=\left[\begin{array}{cc}0 & A_z \\ B_z & 0 \end{array} \right], \ \ \ \Dom(T_z):=\Dom(B_z)\oplus\Dom(A_z).
\end{equation}
Its square is directly related to $A_z B_z$, namely
\beq\label{eq:Tsquared}
T_z^2= \left[\begin{array}{cc}A_z B_z & 0 \\ 0 & B_z A_z \end{array} \right], \ \ \ \Dom(T^2_z)=\Dom(A_z B_z)\oplus\Dom(B_z A_z).
\eeq
A similar idea is used in \cite{HKM}, where results on the relation between $\sp (AB) $ and $\sp (BA)$ are derived. The following lemma will be of use for us.

\begin{lemma}[\cite{HKM}, Lem 2.1]\label{lemma:rsT} Let $A\in\cC(\cH_2,\cH_1)$, $B\in\cC(\cH_1,\cH_2)$  and let $T\in\cC(\cH_1,\cH_2)$ be defined as in (\ref{eq:defT}) (with the subscript denoting dependence on $z$ ommited). Then
\[
\rs\,T=\{\lambda\in\C : \ \lambda^2 \in {\rm rs}(A B)\cap {\rm rs}(B A)\}.
\]
\end{lemma}
\proof `$\subset$': \ Suppose $\lambda\in \rs\, T\setminus\{0\}$. Then one can check that the algebraic inverse of $BA-\lambda^2\one$ equals
\[
\lambda^{-1}{\red\tilde P}_{\cH_2,\cH_1}(T-\lambda\one)^{-1}J_{\cH_2},
\]
hence $BA-\lambda^2\one\in\cC_\inv(\cH_2)$. Analogously we obtain $AB-\lambda^2\one\in\cC_\inv(\cH_1)$.

Suppose now $0\in\rs\, T$. This implies that $A$, $B$ are invertible and consequently $AB$, $BA$ are invertible.

`$\supset$': \ Suppose $\lambda^2\in \rs (AB)\cap\rs (BA)$. 
Obviously, $\lambda^2\in\rs (T^2)$. 

Suppose that $v\in\Ker(T+\lambda\one)$. Then $Tv=-\lambda v\in\Dom T$. Hence $v\in\Dom (T^2)$ and
\[(T^2-\lambda^2\one)v=(T-\lambda\one)(T+\lambda\one)v=0,\]
which implies $v=0$. Hence $T+\lambda\one$ is injective.

Suppose that $w\in\cH_1\oplus\cH_2$. Then there exists $v\in\Dom T^2$ such that $w=(T^2-\lambda^2)v$. But
\[w=(T+\lambda\one)(T-\lambda\one)v.\]
Hence $w\in\Ran(T+\lambda\one)$. 

Thus we have shown that $T+\lambda\one$ is invertible, or $-\lambda\in\rs T$. The same argument shows  $\lambda\in\rs T$. \qeds


{
{\noindent\emph{Proof of Thm. \ref{theorem:holoproduct}. \ }} 
Let $z\mapsto P_z$, resp. $z\mapsto Q_z$ be resolutions of holomorphy of $z\mapsto A_z$, resp. $z\mapsto B_z$. Then
\[
z\mapsto W_z:=\mat{0}{Q_z}{P_z}{0}
\]
is a resolution of holomorphy of $z\mapsto T_z$  defined as in (\ref{eq:defT}).
  By Lemma \ref{lemma:rsT}, $\lambda,-\lambda\in \rs\,T_z$. Define $V_z:=(T_z-\lambda\one)^{-1}W_z$. Clearly, $V_z$ is holomorphic and has values in bounded operators, and so does
\[
T_z^2 V_z= (\one+\lambda (T_z-\lambda\one)^{-1}) T_z W_z.
\]
Moreover, $V_z$ is injective for all $z\in\Theta_0$ and
\begin{eqnarray*}
\Ran V_z &=& (T_z-\lambda\one)^{-1}\Ran W_z  = (T_z-\lambda\one)^{-1}\Dom(T_z)\\ &= &(T_z-\lambda\one)^{-1}(T_z+\lambda\one)^{-1}(\cH_1\oplus\cH_2)=\Dom(T_z^2).
\end{eqnarray*}
Hence $z\mapsto T_z^2$  is holomorphic. Therefore,
 $z\mapsto (\lambda^2\one-T_z^2)^{-1}$ is holomorphic. This implies the holomorphy of $z\mapsto (\lambda^2\one-A_z B_z)^{-1}$,
$z\mapsto (\lambda^2\one-B_z A_z)^{-1}$. \qed}

\subsection{Case \texorpdfstring{$\Dom(A_z)+\Ran B_z=\cH$}{Lg}}

In this section we use a different strategy. The idea is to represent a subspace closely related to ${\rm Gr}(A_z B_z)$ as the kernel of a bounded operator which depends in a holomorphic way on $z$.
This allows us to treat the holomorphy of the product $A_z B_z$ under the assumption that $\Dom(A_z)+\Ran B_z=\cH$.

\begin{theorem}\label{thm:rancase}
Let $\Theta\ni z\mapsto A_z\in\cC(\cH,\cH_2)$ and $\Theta\ni z\mapsto B_z\in\cC(\cH_1,\cH)$ be holomorphic. Suppose that $A_z B_z$ is closed and 
\beq
\Dom(A_z)+\Ran B_z=\cH
\eeq 
for all $z\in\Theta$. Then $\Theta\ni z\mapsto A_z B_z$ is holomorphic.    
\end{theorem}
\proof {Let $U_z$, $V_z$ be resolutions of holomorphy of respectively $A_z$, $B_z$, so in particular $\Ran U_z=\Dom(A_z)$, $\Ran B_z V_z =\Ran B_z$.} Let $S_z$ be defined by 
\[
S_z \left[\begin{array}{c} x \\ y \end{array} \right]=B_z V_z x - U_z y.
\]
Clearly, the function $z\mapsto S_z$ is {\red holomorphic (as a function with values in bounded operators)} and $\ran S_z = \Dom(A_z)+\Ran B_z=\cH$
by assumption. Therefore, $z\mapsto\Ker S_z$ is holomorphic by Prop. \ref{stwi8}.

A straightforward computation shows that
\[
\begin{aligned}
\Gr(A_z B_z)=\left\{ \left[\begin{array}{c} V_z x \\ {A_z U_z y}\end{array} \right] : \ B_z V_z x=U_z y \right\}&= \left[\begin{array}{cc}V_z & 0 \\ 0 & A_z U_z\end{array}\right]\ker\, S_z \\ &=:T_z\, \ker S_z.
\end{aligned}
\]
The function $z\mapsto T_z$ has values in injective bounded operators and is holomorphic, therefore $z\mapsto T_z \ker S_z$ is holomorphic by Prop. \ref{stwi9}.\qeds

An analogous theorem for continuity is proved in \cite[Thm. 2.3]{neubauer}, using however methods which do not apply to the holomorphic case. 

\begin{remark}An example when the assumptions of Thm. \ref{thm:rancase} are satisfied is provided by the case when $A_z$ and $B_z$ are densely defined Fredholm operators. It is well-known that the product is then closed (it is in fact a Fredholm operator), whereas the propriety $\Dom(A_z)+\Ran B_z=\cH$ follows from ${\rm codim}(\Ran B_z)<\infty$ and the density of $\Dom(A_z)$.\end{remark}

\subsection{Sums of closed operators}
Using, for instance, the arguments from Prop. \ref{kjh2}), it is easy to  show
\bep
If $z\mapsto T_z\in \cC(\cH_1,\cH_2)$ and $z\mapsto S_z\in \cB(\cH_1,\cH_2)$ are holomorphic then $z\mapsto T_z+S_z\in \cC(\cH_1,\cH_2)$ is holomorphic.\eep 

To prove a more general statement, we reduce the problem of the holomorphy of the sum to the holomorphy of the product of suitably chosen closed operators. To this end we will need the following easy lemma.

\begin{lemma}\label{lem:TA} A function $z\mapsto T_z\in\cC(\cH_1,\cH_2)$ is holomorphic iff the function
\[
z\mapsto A_z:=\mat{\one}{0}{T_z}{\one}, \quad \Dom(A_z)=\Dom(T_z)\oplus \cH_1 
\]
is holomorphic. Moreover, $\Ran A_z=\Dom(T_z)\oplus \cH_1$.
\end{lemma}
\proof The function $z\mapsto A_z$ is holomorphic iff
\beq\label{eq:Tzmo}
z\mapsto A_z-\one=\mat{0}{0}{T_z}{0}
\eeq
is holomorphic. The claim follows by remarking that the graph of (\ref{eq:Tzmo}) is equal to the graph of $T_z$ up to a part which is irrelevant for the holomorphy.
\qed

\begin{theorem}\label{thm:sum}Let  $\Theta\ni z\mapsto T_z, S_z\in \cC(\cH_1,\cH_2)$ be holomorphic. Suppose that
\beq\label{eq:asum}
\Dom(S_z)+\Dom(T_z)=\cH_1.
\eeq
and $T_z+S_z$ is closed for all $z\in\Theta$. Then $\Theta\ni z\mapsto T_z + S_z$ is holomorphic.
\end{theorem}
\proof Let $A_z$, $B_z$, resp. $C_z$ be defined
 as in Lem. \ref{lem:TA} from $T_z$, $S_z$, resp. $T_z+S_z$. The holomorphy of $T_z+S_z$ is equivalent to holomorphy of $C_z$. An easy computation shows that $C_z=A_z B_z$. By (\ref{eq:asum}) we have 
\[
\Dom(A_z)+\Ran B_z=(\Dom(T_z)\oplus\cH_1)+(\Dom(S_z)\oplus\cH_1)=\cH_1\oplus\cH_1.
\] 
Moreover, $A_z,B_z$ are holomorphic by Lem. \ref{lem:TA}, therefore $A_z B_z$ is holomorphic by Thm. \ref{thm:rancase}.
\qeds


\begin{remark}If $T_z$ is unbounded then the auxiliary operator $A_z$ introduced in Lem. \ref{lem:TA} satisfies $\sp(A_z)=\C$ (cf. \cite[Ex. 2.7]{HKM}). The proof of Thm. \ref{thm:sum} is an example of a situation where even if one is interested in the end in operators with non-empty resolvent set, it is still useful to work with operators with empty resolvent set.
\end{remark}

\end{document}